\documentclass[a4,11pt]{amsart}

\usepackage{latexsym}
\usepackage{amsmath}
\usepackage{amssymb}
\usepackage{amsthm}
\usepackage[mathcal]{eucal}
\usepackage{amsmath, amssymb,enumerate}
\usepackage{mathrsfs}
\usepackage{enumerate}
\usepackage[colorlinks=true, pdfstartview=FitV, linkcolor=blue, citecolor=blue,
 urlcolor=blue]{hyperref}
 \usepackage[normalem]{ulem}
\usepackage{color}
\allowdisplaybreaks

 \makeatletter \@addtoreset{equation}{section}

\makeatother \makeatletter

\newtheorem{theorem}{Theorem}
\newtheorem{proposition}{Proposition}
\newtheorem{lemma}{Lemma}
\newtheorem{remark}{Remark}

\newcommand{\R}{\mathbb{R}}

\newcommand{\C}{\mathbb{C}}

\newcommand{\z}{{\mathbb{R}^N}}
\newcommand{\rn}{{\mathbb{R}^N}}

\newcommand{\ccc}{C_c^\infty(\z\setminus\{0\})}
\newcommand{\f}{\mathfrak{a}_\lambda}

\DeclareMathOperator{\re}{Re}
\DeclareMathOperator{\im}{Im}

\allowdisplaybreaks
\begin{document}
\title[]{Fourth-order Schr\"odinger type operator with unbounded coefficients in $L^2(\z)$}
\author{Federica Gregorio, Cristian Tacelli}
\address{Dipartimento di Matematica, Università degli Studi di Salerno, Fisciano, Italy}
\email{fgregorio@unisa.it}
\email{ctacelli@unisa.it}
\thanks{The authors are members of the Gruppo Nazionale per l’Analisi Matematica, la Probabilità
e le loro Applicazioni (GNAMPA) of the Istituto Nazionale di Alta Matematica (INdAM). This article is based upon work from COST Action CA18232 MAT-DYN-NET, supported by COST (European Cooperation in Science and Technology), www.cost.eu.}

\maketitle

\

\begin{abstract}
In this paper we study generation results in $L^2(\z)$  for the fourth order Schr\"odinger type operator with unbounded coefficients of the form
$$A=a^{2} \Delta  ^2+V^{2}$$ 
where  $a(x)=1+|x|^{\alpha}$ and $V=|x|^{\beta}$ with $\alpha>0$ and $\beta >(\alpha-2)^+$. We obtain that $(-A,D(A))$ generates an analytic strongly continuous semigroup in $L^2(\z)$  for $N\geq5$. Moreover,   the maximal domain $D(A)$ can be characterized for $N>8$ by the weighted Sobolev space
 \[
D_2(A)=\{u\in H^{4}(\R^{N})\,:\,V^{2}u\in L^{2}(\R^{N}), |x|^{2\alpha-h}D^{4-h}u\in L^{2}(\R^{N}) \text{ for } h=0,1,2,3,4\}.
\]  
\end{abstract}

\medskip


\section{Introduction}

Second order elliptic operators with unbounded coefficients and singular or unbounded potentials have been widely investigated,
there exists nowadays a huge literature, see for example 
\cite{for-lor, lor-rhan, for-gre-rha, boutiah-et-al1, bcgt, boutiah-et-al2,BelhajAli, can-rhan-tac1, can-rhan-tac2, can-gre-rhan-tac, can-tac1, met-oka-sob-spi, met-spi2, met-spi-tac, met-spi-tac2, gre-ker, dur-man-tac1} and the references therein. 

Recently, there is an increasing interest towards elliptic operators of higher order. They are involved in  models of elasticity \cite{mele}, condensation in graphene \cite{sgk}, free
boundary problems \cite{adams} and non-linear elasticity \cite{antman}. 

The semigroup generated by a class of higher order elliptic operators with measurable coefficients
has been  systematically studied by Davies in  \cite{davies95a}.   
Unbounded coefficients have also been considered in recent time. In \cite{agrt}, the authors study generation results for the square of the Kolmogorov operator $L=\Delta +\frac{\nabla \mu}{\mu}\cdot\nabla$ in the weighted space $L^2(\z,d\mu)$ giving sufficient conditions for the characterization of  both the domains of    $L$ and $-L^2$. 
Perturbation of elliptic operators by singular or unbounded potentials is a classical problem studied by many authors. The most important features concern  the characterization of the  domain of the sum operator and kernel estimates for the solution. It is known that there is a strong relation between   Schr\"odinger type operators and Hardy's type inequalities. In this sense, relying on Rellich's inequality, in  \cite{gre-mil}, perturbation of the bi-Laplacian operator by a singular potential of the form $c|x|^{-4}, c<\left(\frac{N(N-4)}{4}\right)^2$ has been considered.  Non-existence result for the more general parabolic equation $u_t=-(-\Delta)^mu+c|x|^{-2m}u$ in $\Omega\times(0,T),m\geq1$ and $c$ greater than a suitable Hardy constant has been stated in \cite{gal-kam}. Special classes of  potentials for higher order operators have been investigated such as   Kato class   in \cite{dav-hinK} and reverse H\"older class  in \cite{sugano, liu-don}, see also the references therein. 

In this paper we propose to start the investigation in $L^2(\z)$ for the following Schr\"odinger type operator
$$A=a^{2} \Delta  ^2+V^{2}$$
with unbounded diffusion   $a(x)=1+|x|^{\alpha}$ and potential $ V=|x|^{\beta}$ where the constants $\alpha,\beta$ are positive  and the potential is supercritical, in the sense that $\beta >\alpha-2$. 
In section \ref{generation}, studying the properties of the form associated to $A+\lambda$ where $\lambda\geq\lambda_0>0$ we prove that $(-A,D(A))$ is the generator of an analytic $C_0$-semigroup in $L^2(\z)$ for $N\geq5$. Moreover, in Section \ref{domain} we aim at characterizing the domain $D(A)$ for suitable values of the dimension $N$.  Up to a transformation, we can consider the operator $(-\Delta)^2+\tilde V^2$ where $\tilde V$   belongs to the reverse H\"older class. Making use  of a priori estimates proved in   \cite{sugano} and some weighted higher order interpolation and Calder\'on-Zygmund estimates, for $N>8$, we characterize   the domain $D(A)$ by the weighted Sobolev space 
\[
D_{2}(A)=\{u\in H^{4}(\R^{N})\,:\,V^{2}u\in L^{2}(\R^{N}), |x|^{2\alpha-h}D^{4-h}u\in L^{2}(\R^{N}) \text{ for } h=0,1,2,3,4\}.
\]
The constraints on the dimension $N\geq5$, for the generation result, and $N>8$, for the domain characterisation, correspond to the application of the  Rellich  inequality and the higher order Rellich  inequality, respectively.

Finally, we prove that the spectrum of $A$ is real and consists of a sequence of negative eigenvalues accumulating at $-\infty$.

\section{Generation in $L^2(\z)$}\label{generation}

Consider the bi-Laplacian operator with unbounded diffusion and potential term
$$A=a^{2} \Delta ^2+V^{2},$$
where  $a(x)=1+|x|^{\alpha}$ and $V=|x|^{\beta}$ with $\alpha>0$ and $\beta >(\alpha-2)^+$.
Let $u,v \in 
\ccc$, since

\begin{align}\label{intA}
\int_\z Au\,\overline v\,dx&
	=\int_\z\left[ a^{2}\left(\Delta^{2}u\right) +V^{2}u\right]\overline v\,dx
	=\int_\z\left[ \Delta u\Delta (a^{2}\overline{v})+V^{2}u\overline v \right]\,dx\\
&	=\int_\z\left[ a^{2} \Delta u\Delta \overline{v}+2 \nabla a^{2}\cdot \nabla \overline{v}\Delta u
		+\Delta a^{2} \Delta u\, \overline{v}+V^{2}u\overline v \right]\,dx \nonumber,
\end{align}
we define the following sesquilinear form
\begin{align}\label{eq:form1}
\f(u,v)&=\int_\z\left[ a^{2}\Delta \overline{v}\Delta u+2\nabla a^{2}\cdot \nabla \overline{v}\Delta u
	+\Delta a^{2} \Delta u\,\overline{v}+(V^{2}+\lambda)u\overline v\right]dx\nonumber\\
&=\int_\z \Big\{(1+|x|^\alpha)^{2}\Delta \overline{v}\Delta u
	+4\alpha(1+|x|^{\alpha})|x|^{\alpha-2} x\cdot \nabla \overline{v}\Delta u \\
&\qquad +2 \alpha\left[ (2\alpha-2+N)|x|^{2\alpha-2}
	+ (\alpha-2+N)|x|^{\alpha-2} \right]
		\overline{v}\Delta u 
		+(V^{2}+\lambda)u\overline v\Big\}dx \nonumber
\end{align}
on 
$\ccc$ for some $\lambda \geq\lambda_{0}$, where $\lambda_{0}>0$ will be chosen later. Note that we propose to define $\f$ on $\ccc$ for further purposes that will be clear in the sequel.


In the following we prove that the form $\f$ is densely defined, accretive, continuous and closable, and that the domain of the
closure  coincides with the weighted  Sobolev  space
\[
D=\{u\in H^2(\z)\,:\,(1+|x|^\alpha)\Delta u,\,|x|^{\alpha-1}\nabla u,\,|x|^{\alpha-2}u, Vu \in L^2(\z)\}
\]
endowed with the  norm
\begin{equation}\label{normaD}
\| u\|_{D}=\|(1+|x|^{\alpha})\Delta u\|_{2}+\||x|^{\alpha-1}|\nabla u|\|_{2}+\| |x|^{\alpha-2}u\|_{2}+\|Vu\|_{2}+\|u\|_{2}.
\end{equation}



We start with some estimate that will be useful in the sequel.

\begin{lemma}\label{lem:stimagamma}
Let $N\geq 5$ and $\gamma>0$. For every real function $u\in 
\ccc$, there exists a constant $k\in\R$,  which need not be positive, depending on $\gamma$ and $N$ such that
\begin{equation*}\label{eq:stimagamma}
\int_\z|x|^\gamma(\Delta^2 u)u\,dx\geq k\int_\z|x|^{\gamma-4}u^2\,dx.
\end{equation*}
\end{lemma}

\begin{proof}
Let us first compute the derivatives of $|x|^\gamma$
\begin{align*}
\nabla |x|^\gamma&=\gamma|x|^{\gamma-2}x\\
 D_{ij}|x|^\gamma&=\gamma(\gamma-2)|x|^{\gamma-4}x_{i}x_{j}+\gamma|x|^{\gamma-2}\delta_{ij}\\
\Delta|x|^\gamma&=\gamma(\gamma-2+N)|x|^{\gamma-2}=:c_1|x|^{\gamma-2}\\
\Delta^2|x|^\gamma&=\gamma(\gamma-2)(\gamma-2+N)(\gamma-4+N)|x|^{\gamma-4}=:c_2|x|^{\gamma-4}
\end{align*}
and observe that $c_1$ is positive and the sign of $c_2$ depends on whether   $\gamma$ is bigger than 2 or not.

 We  now want to estimate $\int_\z|x|^\gamma(\Delta^2 u)u\,dx$. Integrating by parts one has that
\begin{align*}
\int_\z|x|^\gamma(\Delta^2 u)u\,dx&=\int_\z |x|^\gamma (\Delta u)^2\,dx+2\int_\z \nabla |x|^\gamma\cdot \nabla {u}\Delta u\,dx+\int_\z\Delta|x|^\gamma{u}\,\Delta u\,dx.
\end{align*}

The second term of the right-hand side can be written as
\begin{align*}
 2\int_\z \Delta u \nabla u\cdot\nabla |x|^\gamma\,dx&=-2\int_\z \nabla u\cdot \nabla \left(\nabla u\cdot \nabla |x|^\gamma  \right)\,dx
	=-2\sum_{i,j=1}^{N}\int_\z D_{i}uD_{i}\left( D _{j}uD_{j}|x|^\gamma \right)\,dx\\
& =-2\sum_{i,j=1}^{N}\int_\z D_{i}uD_{ij} uD_{j}|x|^\gamma\,dx -2\sum_{i,j=1}^{N}\int_\z D_{i}u D _{j}uD_{ij}|x|^\gamma\,dx \\
&=-\sum_{i,j=1}^{N}\int_\z D_{j} \left(D_{i}u\right)^{2}D_{j}|x|^\gamma\,dx-2\int_\z (D^{2}|x|^\gamma)\nabla u\cdot \nabla u\,dx\\
&  =\int_\z |\nabla u|^{2}\Delta |x|^\gamma\,dx-2\int_\z (D^{2}|x|^\gamma)\nabla u\cdot \nabla u\,dx
\end{align*}
and since $\Delta u^{2}=2u\Delta u+2|\nabla u|^{2}$, the third term can be  written as
\begin{align*}
\int_\z u\Delta u\Delta |x|^\gamma\,dx&=\frac{1}{2}\int_\z \Delta u^{2}\Delta |x|^\gamma\,dx-\int_\z |\nabla u|^{2} \Delta |x|^\gamma\,dx\\
& = \frac{1}{2}\int_\z u^{2}\Delta^{2}|x|^\gamma\,dx-\int_\z |\nabla u|^{2} \Delta |x|^\gamma\,dx.
\end{align*}
Hence, one obtains
\begin{align}\label{stima}
\int_\z |x|^\gamma\left(\Delta^{2}u\right)u\,dx&=\int_\z|x|^\gamma \left( \Delta u \right)^{2}\,dx-2\int_\z (D^{2}|x|^\gamma)\nabla u\cdot \nabla u\,dx
	\nonumber\\&\quad+\frac{1}{2}\int_\z u^{2}\Delta^{2}|x|^\gamma\,dx.
\end{align}
Taking into account the  derivatives of $|x|^\gamma$ the equality \eqref{stima} reads as
\begin{align*}
\int_\z |x|^\gamma\left(\Delta^{2}u\right)u\,dx &=\int_\z |x|^\gamma\left( \Delta u \right)^{2}\,dx
-2 \gamma(\gamma-2)\int_\z      |x|^{\gamma-4} \sum_{i,j=1}^{N}x_{i}x_{j}D_{i}uD_{j}u\,dx\\&\quad
-2\gamma\int_\z  |x|^{\gamma-2}|\nabla u|^{2}\,dx 
	 +\frac{1}{2} c_{2}\int_\z |x|^{\gamma-4}u^{2}\,dx.
\end{align*}
We can estimate 
$$\sum_{i,j=1}^{N}x_{i}x_{j}D_{i}uD_{j}u=\left( \langle x,\nabla u\rangle  \right)^{2}\leq |x|^{2}|\nabla u|^2,$$
then $-2\gamma( \gamma-2)\left( \langle x,\nabla u\rangle  \right)^{2} \geq -2 \gamma|\gamma-2|  |x|^{2}|\nabla u|^2$.
Therefore, one obtains
\begin{align*}
\int_\z |x|^{\gamma}\left(\Delta^{2}u\right)u\,dx &\geq \int_\z |x|^{\gamma}\left( \Delta u \right)^{2}\,dx
	-2\gamma\left(|\gamma-2|+1\right)\int_\z  |x|^{ \gamma-2}|\nabla u|^{2}\,dx\\\nonumber
	&\quad+\frac{1}{2}c_{2}\int_\z |x|^{ \gamma-4}u^{2}\,dx.
\end{align*}
Since
\begin{align*}
\int_\z |x|^{\gamma-2}|\nabla u|^{2}dx&=-\int_\z u\,{\rm div}\left( |x|^{\gamma-2}\nabla u \right)dx
	\\&=-\int_\z |x|^{\gamma-2}u\Delta udx-\frac12(\gamma-2)\int_\z  |x|^{\gamma-4}x\cdot\nabla u^{2}dx \\
&\quad =-\int_\z |x|^{\gamma-2}u\Delta udx+\frac{\gamma-2}{2}\left( \gamma-4+N \right)\int_\z |x|^{\gamma-4}u^{2}dx,
\end{align*}
setting $c_3=2\gamma(|\gamma-2|+1)+1$, $c_4=c_3\frac{\gamma-2}{2}(\gamma-4+N)$, $c_5=2^{-\frac12}$ and considering two suitable constants $c_6,k\in\R$ one has
\begin{align}\label{stima1}
\int_\z |x|^{\gamma}\left(\Delta^{2}u\right)u\,dx &\geq  \frac12\int_\z |x|^\gamma(\Delta u)^2dx+\int|x|^{\gamma-2}|\nabla u|^2dx\nonumber\\&\quad+\frac12\int_\z |x|^\gamma(\Delta u)^2dx+  c_3\int_\z|x|^{\gamma-2}u\Delta udx+\left(\frac{c_2}{2}-c_4\right)\int_\z|x|^{\gamma-4}u^2dx\nonumber\\ 
&=\frac12\int_\z |x|^\gamma(\Delta u)^2dx+\int|x|^{\gamma-2}|\nabla u|^2dx\\&\quad+\int_{\z}\left(c_5|x|^{\frac{\gamma}{2}}\Delta u+  c_6|x|^{\frac{\gamma}{2}-2}u\right)^2dx+k\int_\z|x|^{\gamma-4}u^2dx\nonumber
\end{align}
where we have rearranged the terms for future convenience.

Hence,  
$
\int_\z |x|^{\gamma}\left(\Delta^{2}u\right)u\,dx  \geq   k\int_\z|x|^{\gamma-4}u^2dx.$
 
\end{proof}

\begin{proposition}\label{pr:accret}
 Let $N\geq5,\alpha>0,\beta>(\alpha-2)^+$. The form $\f$ is accretive. Moreover,  for every $u\in 
\ccc$ the following inequality holds
\begin{equation*}\label{eq:accretiv2}
\re\f(u,u)\geq 
\frac{1}{4}\|(1+|x|^{\alpha})\Delta u\|_{2}^{2}+\||x|^{\alpha-1}|\nabla u|\|_{2}^{2}+\| |x|^{\alpha-2}u\|_{2}^{2}
	+\frac{1}{2}\|Vu\|_{2}^{2}+\|u\|_{2}^{2}.
\end{equation*}

\end{proposition}
\begin{proof}

If $u$ is a real function, by \eqref{intA} we can write
\begin{align*}
\f(u,u)&=\int_\z (1+|x|^\alpha)^2 (\Delta^2 u) u\,dx+
		\int_\z\left( V^{2}+\lambda \right)u^{2}dx\\
		&=\int_\z   (\Delta   u)^2\,dx+2\int_\z |x|^{\alpha} (\Delta^2 u) u\,dx+\int_\z |x|^{2\alpha} (\Delta^2 u) u\,dx+
		\int_\z\left( V^{2}+\lambda \right)u^{2}dx.
\end{align*}

By Lemma \ref{lem:stimagamma}   there exists a constant such that the term $2\int_\z |x|^{\alpha} (\Delta^2 u) u\,dx\geq k_1\int_\z|x|^{\alpha-4}u^2dx.$ Moreover,  one can estimate the term $\int_\z |x|^{2\alpha} (\Delta^2 u) u\,dx$ as in \eqref{stima1} to obtain that
\begin{align*}
\f(u,u)&\geq\frac12\int_\z   (\Delta   u)^2\,dx+\frac12\int_\z   (\Delta   u)^2\,dx+k_1\int_\z|x|^{\alpha-4}u^2dx+\frac12\int_\z|x|^{2\alpha}(\Delta u)^2dx\\
     &\quad +\int_\z|x|^{2\alpha-2}|\nabla  u|^2dx+k_2\int_\z|x|^{2\alpha-4}  u^2dx +\int_\z\left( |x|^{2\beta}+\lambda \right)u^{2}dx    \\
&\geq\int_\z \left[\frac{1}{4}(1+|x|^{\alpha})^{2}(\Delta u)^{2}
	+|x|^{2\alpha-2}|\nabla u|^{2}+\frac{1}{2}(\Delta u)^{2}\right]dx\nonumber\\
&\qquad +  \int_\z \left[k_1   |x|^{\alpha-4}+k_2|x|^{2\alpha-4}  + |x|^{2\beta}+\lambda \right]u^{2}dx.
\end{align*}

Finally, by Rellich's inequality, there exists a positive $c_0$ such that
\begin{align}\label{eq:accretivity}
 \f(u,u)&\geq  	 \int_\z \left[\frac{1}{4}(1+|x|^{\alpha})^{2}(\Delta u)^{2}
	+|x|^{2\alpha-2}|\nabla u|^{2}\right]dx\nonumber\\
&\qquad+	\int_\z \left(\frac{c_{0}}{|x|^{4}}+k_{1} |x|^{\alpha-4}+k_{2} |x|^{2\alpha-4}+|x|^{2\beta}+\lambda\right)u^{2}dx\nonumber\\
&\quad = \int_\z \left[\frac{1}{4}(1+|x|^{\alpha})^{2}(\Delta u)^{2}
	+|x|^{2\alpha-2}|\nabla u|^{2}+|x|^{2\alpha-4}u^{2}+\frac{1}{2}|x|^{2\beta}u^{2}+u^{2}\right]dx\nonumber\\
&\qquad+	\int_\z \left(\frac{c_{0}}{|x|^{4}}+k_{1} |x|^{\alpha-4}+(k_{2}-1) |x|^{2\alpha-4}
		+\frac{1}{2}|x|^{2\beta}+(\lambda-1)\right)u^{2}dx\nonumber\\
&\quad \geq \int_\z\left[ \frac{1}{4}(1+|x|^{\alpha})^{2}(\Delta u)^{2}
	+|x|^{2\alpha-2}|\nabla u|^{2}+|x|^{2\alpha-4}u^{2}+\frac{1}{2}|x|^{2\beta}u^{2}+u^{2}\right]dx		\geq 0
\end{align}
for some $\lambda\geq\lambda_{0}$ since  the growth at infinity and in 0 is led by the terms $\frac12|x|^{2\beta}$ and $c_0|x|^{-4}$, respectively, by the relation $2\beta>2\alpha-4>\alpha-4>-4$.

If $u$ is complex valued then one can write $u=\re u+i\im u$ and in view of $\re \f(u,u)=\f(\re u,\re u)+\f(\im u,\im u)$  the thesis follows.
\end{proof}

Since the form $\f$ is accretive on $\ccc$ 
 we can associate to $\f$ the norm
\[
\|u\|_{\f}^{2}=\re\f(u,u)+\|u\|^{2}_{2}.
\]
By Proposition \ref{pr:accret}  it follows that
\begin{equation}\label{eq:accretiv1}
\|u\|_{D}\leq C\|u\|_{\f}\text{ for all } u \in 
\ccc.
\end{equation}

We are now ready to study the properties of the form $\f$.

\begin{proposition}
 Let $N\geq5,\alpha>0,\beta>(\alpha-2)^+$. The form $\f$ si densely defined, accretive and continuous on 
$\ccc$.
\end{proposition}
\begin{proof}  $\ccc\subset D$ then $\f$ is densely defined. The accretivity is stated in \eqref{eq:accretivity}.
As regards the continuity we observe that from \eqref{eq:form1} we have 
\begin{align}\label{eq:continuity}
|\f(u,v)|&\leq \|(1+|x|^{\alpha})\Delta u\|_{2}\|(1+|x|^{\alpha})\Delta  v\|_{2} 
	+4\alpha\|(1+|x|^{\alpha})\Delta u\|_{2} \||x|^{\alpha-1}\nabla  v\|_2\nonumber \\
& \quad +2\alpha(2\alpha-2+N)\| |x|^{\alpha}\Delta u\|_{2}\||x|^{\alpha-2}v\|_{2}\nonumber \\
&\quad  +2\alpha(\alpha-2+N)\|\Delta u\|_{2}\||x|^{\alpha-2}v\|_{2}+\|Vu\|_{2}\|Vv\|_{2}+\lambda \|u\|_{2}\|v\|_{2}\nonumber \\
& \leq C\|u\|_{D}\|v\|_{D}.
\end{align}
Then, by \eqref{eq:accretiv1}  
\begin{equation*}\label{eq:stima-auv}
|\f(u,v)|\leq C\|u\|_{\f}\|v\|_{\f}.
\end{equation*}
\end{proof}

Combining \eqref{eq:accretiv1}   and \eqref{eq:continuity}  we can deduce that the norms $\|\cdot\|_{\f}$ and $\|\cdot\|_{D}$ are equivalent on 
$\ccc$.

Now we prove the closability of the form $\f$.

\begin{proposition}
 Let $N\geq5,\alpha>0,\beta>(\alpha-2)^+$. The form $\f$ is closable on 
$\ccc$.
\end{proposition}
\begin{proof}
Let $(u_n)\in 
\ccc$ be such that $u_{n}\to 0$ in $L^{2}(\rn)$ and $\|u_n-u_m\|_{\mathfrak{a}_\lambda}\to 0$ as $n,m\to \infty$.
By estimate \eqref{eq:accretiv1} it follows that $(u_n)$ and $(\Delta u_{n})$ are   Cauchy sequences in
  $L^{2}(\R^{N})$. Using the interpolation inequality
\[
\|\nabla u\|_{2}\leq C\|\Delta u\|_{2}\|u\|_{2}
\]
it follows  that $(u_{n})$ is a Cauchy sequence in 
 $H^2(\R^N)$ and then it converges to $0$ in  $H^2(\R^N)$.

Taking into account the expression  \eqref{normaD} for the norm $\|\cdot\|_D$, we have that  
$w_{n}=(1+|x|^{\alpha})\Delta u_{n}$ and $v_n=|x|^{\alpha-1}\nabla u_{n}$ are 
Cauchy sequences in $L^2(\R^N)$.
Since they converge to $0$ a.e. they also converges to $0$ in $L^2(\R^N)$ and one has

 $$\int_{\R^N} (1+|x|^\alpha)|\Delta u_n|^2\,dx\to 0\text{ and }
 \int_{\R^N} |x|^{\alpha-1}|\nabla u_n|^2\,dx\to 0$$ as $n\to \infty$.
Moreover  \eqref{normaD}  gives also that 

$$\int_{\R^N} |x|^{2\alpha-4}|u_n|^2\,dx\to 0$$
and 
$$\int_{\R^N} |x|^{2\beta}|u_n|^2\,dx\to 0$$ as $n\to \infty$.

Then we have proved that $\|u_{n}\|_{D}\to 0$ as $n\to \infty$.
Since by \eqref{eq:continuity} it follows that
\[
|\f(u_{n},u_{n})|\leq C\|u_{n}\|_{D}^{2},
\]
we have that $\f(u_{n},u_{n})\to 0$ as $n\to \infty$ and then $\f$ is closable.
\end{proof}

Next proposition gives a characterization of the domain of the closure of the form.

\begin{proposition}\label{pr:core-for-a}
 Let $N\geq5,\alpha>0,\beta>(\alpha-2)^+$. The domain of the closure of $\f$ coincides with the weighted Sobolev  space
\[
D=\{u\in H^2(\z)\,:\,(1+|x|^\alpha)\Delta u,\,|x|^{\alpha-1}\nabla u,\,|x|^{\alpha-2}u, Vu \in L^2(\z)\}.
\]
\end{proposition}
\begin{proof}
We recall that
$D$ is a Banach space with respect to the norm $\|\cdot \|_{D}$.
Since 
$\ccc\subset D$  and since the norm $\|\cdot \|_{D}$ is equivalent to the form norm 
$\|\cdot \|_{\f}$ on 
$\ccc$,
we have just to prove that 
$\ccc$ is dense in $D$ with respect the norm $\|\cdot \|_{D}.$

It is enough to prove that the set of functions in $H^2(\R^N)$ with compact support contained in 
$\rn\setminus\{0\}$ is dense in $D$.
Take $u\in D$ and consider $u_n=u\varphi _n$,
where $\varphi_n\in 
\ccc$ is such that 

\begin{equation*}  
\left\{
\begin{array}{ll}
\varphi_n=0  \  \text{in} \   B({\frac1n})\cup B^c({2n}), \\
\varphi_n=1 \  \text{in} \  B(n)\setminus B({\frac2n)}, \\
0\leq \varphi_n\leq 1,\\
|\nabla \varphi_n(x)|\leq C\frac{1}{|x|},\\
|D^{2} \varphi_n(x)|\leq C\frac{1}{|x|^{2}}.\\
\end{array}\right.
\end{equation*}  
Observe that such a function exists. Indeed, let $\varphi \in C^{\infty}([0,+\infty))$ such that $0\leq \varphi(t)\leq 1$, $\varphi(t)=1$ for $0\leq t\leq 1$,
$\varphi(t)=0$ if $t\geq 2$.  Then the function
\begin{equation*}  
\varphi_{n}(x)=
\begin{cases}
1-\varphi(n|x|)  &  \text{in} \   B({\frac2n}), \\
1 &  \text{in} \  B(n)\setminus B({\frac2n)}, \\
\varphi \left( \frac{|x|}{n} \right) &  \text{in} \   B^{c}(n) \\
\end{cases} 
\end{equation*}   
satisfies the desired properties.


We show now that $u_n\to u$ with respect to the norm  $\| \cdot \|_{D}$.
Indeed, $u_n|x|^{\beta}\to u|x|^{\beta}$ and $u_n|x|^{\alpha-2}\to u|x|^{\alpha-2}$ in $L^2(\R^N)$ by dominated convergence.

As regards the first order term we observe that

\begin{eqnarray*}
|x|^{\alpha-1}|\nabla u_n-\nabla u| &=&
|x|^{\alpha-1}\left|(\varphi_n -1)\nabla u + u\nabla \varphi_n\right|\\
&\le & |x|^{\alpha-1}(1-\varphi_n )|\nabla u|+ |x|^{\alpha-1}|u||\nabla \varphi_n|.
\end{eqnarray*}
Therefore, it suffices to prove that the second term of the right-hand side of the above estimate converges to 0 in $L^2(\R^N)$. Indeed it converges a.e. to $0$,
and since
\begin{align*}
 |x|^{\alpha-1}|u|  |\nabla \varphi_n|\leq
 C |x|^{\alpha-2}|u|\chi_{K_n}\in L^{2}(\rn)
\end{align*}
where $K_n=B({\frac2n)}\setminus B({\frac1n)}\cup B({2n})\setminus B(n)$, the convergence is
in $L^{2}(\rn)$.

As regards the diffusion term we can argue in a similar way, since
\begin{align*}
&(1+|x|^{\alpha})|\Delta u_{n}-\Delta u|
	= (1+|x|^{\alpha})|\varphi_{n}\Delta u+2\nabla \varphi_{n}\cdot \nabla u+u\Delta \varphi_{n}-\Delta u| \\
&\qquad \leq \left( 1+|x|^{\alpha} \right)(1-\varphi_{n})|\Delta u|+2(1+|x|^{\alpha})|\nabla u||\nabla \varphi_{n}|
			+(1+|x|^{\alpha})|u||\Delta \varphi_{n}|,
\end{align*}
and 
\[
2(1+|x|^{\alpha})|\nabla u||\nabla \varphi_{n}|\leq C\left( \frac{1}{|x|}|\nabla u|+|x|^{\alpha-1}|\nabla u| \right)\chi_{K_n}
\in L^{2}(\rn)
\]
by Hardy's inequality, and 
\[
(1+|x|^{\alpha})|u||\Delta \varphi_{n}|\leq C\left( \frac{1}{|x|^{2}}| u|+|x|^{\alpha-2}|u| \right)\chi_{K_n}
\in L^{2}(\rn)
\]
by Rellich's inequality.
\end{proof}

We observe that $C_{c}^{\infty}(\rn)\subset D$ and therefore we have that also $C_{c}^{\infty}(\rn)$ is a core for 
$\f$.

The form $\f$ defined in \eqref{eq:form1}  is densely defined, accretive, continuous and closable. Therefore, its closure $\overline\f$ is associated to a closed  operator $(A_\lambda,D(A_\lambda))$ on $L^2(\R^N)$ defined by

\begin{align*}
D(A_\lambda):&=\{u\in D\,:\,\exists\,v\in L^2(\R^N)\,s.t.\,\,\overline\f(u,h)=\langle v,h\rangle ,\,\forall\,h\in D\}\\ A_\lambda u:&=v.
\end{align*}

Taking into consideration the properties of the form $\f$ and \cite[Proposition 1.51 and Theorem 1.52]{ouhabaz}, we obtain the following generation theorem on $L^2(\R^N)$.  

\begin{theorem}\label{thmform}
 Let $N\geq5,\alpha>0,\beta>(\alpha-2)^+$.
The operator $(-A_\lambda, D(A_\lambda))$ generates  a strongly continuous analytic contraction semigroup on $L^2(\R^N)$.
\end{theorem}

Since for $u\in\ccc$ by \eqref{intA} one has that
\[\f(u,v)=\int_\z(Au+\lambda u)\overline{v}\,dx,\qquad\textrm{for every}\ v\in\ccc\]
and $\ccc$ is a core for $\overline\f$, the equality holds for every $v\in D$. Then,  the operator $(-A_\lambda, D(A_\lambda))$ is an extension of $(-A,\ccc)$, more precisely, $Au+\lambda u=A_\lambda u$.
Therefore, there is $(-A,D(A))$, an extension of $(-A,\ccc)$, that generates the analytic $C_0$-semigroup $e^{-tA}=e^{\lambda t}e^{-t A_\lambda}$ in $L^2(\z)$. In the next section we aim at characterizing the domain $D(A)$.

\section{Domain Characterization}\label{domain}
In this section we aim at characterizing $D(A)$ for suitable values of  the dimension $N$.
By definition
\begin{align*}
&D(A)
	=\{u\in D \,:\,\exists\,v\in L^2(\R^N)\,s.t.\,\,\overline\f(u,h)=\langle v,h\rangle ,\,
		\forall\,h\in D \}\\
&	 Au+\lambda u=v.
\end{align*}
where
\[
D=\{u\in H^2(\z)\,:\,(1+|x|^\alpha)\Delta u,\,|x|^{\alpha-1}\nabla u,\,|x|^{\alpha-2}u, Vu \in L^2(\z)\}.
\]

Let $u\in D(A)$. Since $u\in H^{2}(\R^{N})$ one can integrate by parts and obtain  
\[
\int_\z Auv\,dx=\int_\z uA^{*}v\,dx
\]
for every $v \in C_{c}^{\infty}(\R^{N}\setminus\{0\})$ where
\begin{align*}
& A^{*}v=a^{2}\Delta^{2}v + 4\nabla a^{2}\cdot \nabla \Delta v+2\Delta a^{2}\Delta v + 4 Tr D^{2}a^{2}D^{2}v \\
&\qquad +4\nabla \Delta a^{2}\cdot \nabla v + \Delta^{2}a^{2}v +V^{2}v.
\end{align*}
By local elliptic regularity, see \cite{agm}, since the coefficients of $A$ are bounded in every compact set contained in
$\R^{N}\setminus\{0\}$ we have that $u \in H^{4}_{\rm loc}(\R^{N}\setminus\{0\}).$

Now let $f=Au$, then $f\in L^{2}(\R^{N})$ and dividing by $a^{2}$ one obtains
\[
 \Delta  ^{2}u+\left( \frac{V}{a} \right)^{2}u=\frac{f}{a^2}\in L^{2}(\R^{N}).
\]
Thus, consider the Schr\"odinger operator $\tilde A= \Delta   ^{2}u+\tilde V^2 u$ where
$\tilde V= \frac{V}{a}$.
In \cite{sugano} the author established maximal estimates   for $\tilde A$, assuming that the potential  $\tilde V$   belongs to the reverse H\"older class $B_q$ for some $q\geq \frac{N}{2}$.
We recall that 
a nonnegative locally $L^q$-integrable function $\tilde V$ on $\R^N$ is
said to be in $B_q,\,1 < q < \infty$, if there exists $C > 0$ such that the reverse
H\"older inequality
\[
\left( \frac{1}{|B|}\int_B\tilde V^q(x) dx \right) ^{1/q}\leq C\left( \frac{1}{|B|}\int_B\tilde V(x) dx \right)
\]
holds for every $x \in \mathbb{R}^N$ and for every ball $B$ in $\R^N$.

  In the same way as in \cite[Section 3]{boutiah-et-al1} one verifies that 
\begin{align*}\label{eq:stime-v}
C_1(1+|x|^{\beta-\alpha}) & \leq \tilde V \leq C_2 ({1+|x|^{\beta-\alpha}})\quad\text{ if }\beta\geq \alpha ,\\
C_3\frac{1}{1+|x|^{\alpha-\beta}} &\leq \tilde V \leq C_4 \frac{1}{1+|x|^{\alpha-\beta}}\quad\text{ if }\alpha-2<\beta<\alpha, \nonumber
\end{align*}
for some positive constants $C_1$, $C_2$, $C_3$, $C_4$. Then it   can be proved that
$\tilde V\in B_q$ if $\beta -\alpha>-\frac{N}{q}$. In our case, since $\beta>\alpha-2$, the potential $\tilde V\in B_{\frac{N}{2}}$.

The following theorem was proved by Sugano \cite[Theorem 1]{sugano}.
\begin{theorem}
Let $H=(-\Delta)^2+V^2, f\in L^p(\z)$ and $j=0,1,2,3$. Suppose $V\in B_{\frac{N}{2}}$ and there exists a constant $C$ such that $V(x)\leq C m(x,V)^2$, then there exist constants $C_j$ such that 
\[ ||V^{2-\frac{j}{2}}\nabla^j H^{-1}f||_{p}\leq C_j||f||_{p}\quad {\rm for}\  1< p\leq\infty,\]
where $\frac{1}{m(x,V)}=\sup\left\{r>0:\frac{r^2}{|B_r(x)|}\int_{B_r(x)}V(y)\,dy\leq1\right\}$. Moreover, there exists $C'$ such that
\[||\nabla^4 H^{-1}f||_{p}\leq C'||f||_{p}\]
for $1<p<\infty$.
\end{theorem}

In our case $\tilde V\in B_{\frac{N}{2}}$ and since $\tilde V$ behaves as a non-negative polynomial,  the condition $\tilde V(x)\leq C m(x,\tilde V)^2$ is fulfilled , cf. \cite[Remark 5]{sugano}. Therefore it holds $\tilde V^2 u\in L^{2}(\z)$ and $ \Delta ^{2} u \in L^{2}(\R^{N})$ with
\[  \| \tilde V^2 u\|_{2}\leq C\left\|\frac{f}{a}\right\|_{2},\quad 
 \| \Delta ^{2}u\|_{2}\leq C\left\|\frac{f}{a}\right\|_{2}\] for some constant $C>0.$
Then $u \in H^{4}(\R^{N})$ and
\[
D(A)=\{u\in H^{4}(\R^{N})\cap D\,: Au \in L^{2}(\R^{N})\}.
\]
\begin{proposition}\label{pr:pot-estimate}  Let $N\geq5,\alpha>0,\beta>(\alpha-2)^+$. There exists $\lambda_{0}>0$ such that for every $\lambda\geq\lambda_{0}$ and for every $u\in C_{c}^{\infty}(\R^{N}\setminus \{0\})$ the following estimate holds
\begin{equation*}
\|V^{2}u\|_2\leq C \|Au+\lambda u\|_2 .
\end{equation*}
\end{proposition}
\begin{proof}
Let $u\in C_{c}^{\infty}(\R^{N}\setminus \{0\})$ be a real function. 
Recall that by Lemma \ref{lem:stimagamma}, for every $\gamma>0$ we have
\begin{equation}\label{1}
\int_{\R^{N}} |x|^{\gamma}(\Delta ^{2}u)u\,dx\geq 
		k\int_\z |x|^{\gamma-4}u^{2}dx,
\end{equation}
where $k$ is a real constant which depends on $\gamma$ and $N$.
Now we evaluate $ \f(u,(1+V^2)u) =\int_\z (A+\lambda )u (1+V^{2})u\,dx$. By \eqref{1} and  Rellich's inequality there exist $k_1,k_2,k_3,k_4,k_5\in\R$ and a positive $c_0$ such that
\begin{align*}
&\int_{\R^{N}} \left ((1+|x|^{\alpha})^{2}\Delta^{2}u+|x|^{2\beta}u +\lambda u \right)(1+|x|^{2\beta})udx\\
&\quad= \int_{\R^{N}}(\Delta u)^2+ (2|x|^{\alpha}+|x|^{2\alpha}+|x|^{2\beta}+2|x|^{\alpha+2\beta}+|x|^{2(\alpha+\beta)})(\Delta^{2}u)u\\
&\qquad		+(|x|^{4\beta}+(1+\lambda)|x|^{2\beta}+\lambda)u^2dx\\
&\quad \geq \int_{\R^{N}}\left(c_{0}|x|^{-4}+k_{1}|x|^{\alpha-4}+k_{2}|x|^{2\alpha-4}+k_{3}|x|^{2\beta -4}+k_{4}|x|^{\alpha+2\beta-4}
		+k_5|x|^{2(\alpha+\beta-2)}\right.\\
&\qquad \left. +|x|^{4\beta}+(1+\lambda)|x|^{2\beta}+\lambda\right)u^2dx\\
&\quad = \int_{\R^{N}}\left(c_{0}|x|^{-4}+k_{1}|x|^{\alpha-4}+k_{2}|x|^{2\alpha-4}+k_{3}|x|^{2\beta -4}+k_{4}|x|^{\alpha+2\beta-4}
		+k_5|x|^{2(\alpha+\beta-2)}\right.\\
&\qquad \left. +\frac{3}{4}|x|^{4\beta}+(\frac{1}{2}+\lambda)|x|^{2\beta}+\lambda-\frac{3}{4}\right)u^2dx
	+\int_{\R^{N}}\left( \frac{1}{4}|x|^{4\beta}+\frac{1}{2}|x|^{2\beta}+\frac{1}{4}\right)u^{2}dx.
\end{align*}
We can choose $\lambda_0 $ such that
\begin{align*}
&c_{0}|x|^{-4}+k_{1}|x|^{\alpha-4}+k_{2}|x|^{2\alpha-4}+k_{3}|x|^{2\beta -4}+k_{4}|x|^{\alpha+2\beta-4} \\
&\quad+k_5|x|^{2(\alpha+\beta-2)}+\frac{3}{4}|x|^{4\beta}+\left(\frac{1}{2}+\lambda_0\right)|x|^{2\beta}+\lambda_0-\frac{3}{4}\geq 0
\end{align*}
and then
\[
\int_{\R^{N}}(Au+\lambda u)(1+V^{2})udx\geq C \int_{\R^{N}}(1+V^{2})^{2}u^{2}dx
\]
for every $\lambda\geq\lambda_0$. By H\"older's inequality
\[
\|(1+V^{2})u\|^{2}_{2}\leq C\|Au+\lambda u\|_{2}\| (1+V^{2})u\|_{2}
\]
so that
\[
\|V^{2}u\|_{2}\leq \|(1+V^{2})u\|_{2}\leq C\|Au+\lambda u\|_{2}.
\]
\end{proof}

For every function $u$ with derivative up to order 4 we set
\begin{align*}
& |D^{4}u|=\left( \sum_{i,j,k,l=1}^{N}|D_{ijkl}u|^{2} \right)^{\frac12},\,  |D^{3}u|=\left( \sum_{i,j,k=1}^{N}|D_{ijk}u|^{2} \right)^{\frac12},\\
& |D^{2}u|=\left( \sum_{i,j=1}^{N}|D_{ij}u|^{2} \right)^{\frac12},\,  |D u|=\left( \sum_{i=1}^{N}|D_{i}u|^{2} \right)^{\frac12}.
\end{align*}
   By applying  \cite[Lemma 4.4]{Met-Sob-Spi2} to the derivatives of $u$,     for every $u\in C_{c}^{\infty}(\R^{N}\setminus \{0\})$, $h=1,2,3$ and $\gamma\in\R$ the following weighted interpolation inequalities holds 
\begin{equation}\label{eq:iterp-weighted1}
\||x|^{\gamma} D^{h}u\|_2\leq \varepsilon \||x|^{\gamma+1} D^{h+1}u\|_2+C_{\varepsilon}\||x|^{\gamma-1}D^{h-1}u\|_2.
\end{equation}
 

\begin{proposition} Let $N>8,\alpha>0,\beta>(\alpha-2)^+$. For every $u\in C_{c}^{\infty}(\R^{N}\setminus \{0\})$ we have
\begin{align*}
& \| |x|^{2\alpha-h}D^{4-h}u\|_2\leq C\left( \|Au\|_2+\|u\|_2 \right)
\end{align*}
for $h=0,1,2,3,4$.
\end{proposition}
\begin{proof}
Let $u\in C_{c}^{\infty}(\R^{N}\setminus \{0\})$. We propose now to use repeatedly \eqref{eq:iterp-weighted1}   to prove that
\begin{align}
& \||x|^{2\alpha-1}D^{3}u\|_2 \leq \varepsilon  \||x|^{2\alpha}D^{4}u\|_2+ C\||x|^{2\alpha-4} u\|_2 \label{eq:2alpha-1}\\
& \||x|^{2\alpha-2}D^{2}u\|_2 \leq \varepsilon  \||x|^{2\alpha}D^{4}u\|_2+ C\||x|^{2\alpha-4} u\|_2 \label{eq:2alpha-2} \\
& \||x|^{2\alpha-3}Du\|_2 \leq \varepsilon  \||x|^{2\alpha}D^{4}u\|_2+ C\||x|^{2\alpha-4} u\|_2. \label{eq:2alpha-3}
\end{align}
In the sequel we will  denote the constants as $C$ and $\varepsilon $ even if they may change from line to line.
We have
\begin{align*}
 \| |x|^{2\alpha-3}Du\|_2&\leq \varepsilon \||x|^{2\alpha-2}D^{2}u\|_2+C\||x|^{2\alpha-4}u\|_2\\
& \leq \varepsilon \left( \varepsilon \||x|^{2\alpha-1}D^{3}u \|_2+C\||x|^{2\alpha-3}Du \|_2 \right)+C\||x|^{2\alpha-4}u\|_2
\end{align*}
then 
\begin{equation}\label{eq:2alpha-3.1}
\| |x|^{2\alpha-3}Du\|_2\leq \varepsilon \||x|^{2\alpha-1}D^{3}u\|_2+C\||x|^{2\alpha-4}u\|_2.
\end{equation}
Now 
\begin{align*}
& \||x|^{2\alpha-1}D^{3}u\|_2\leq \varepsilon  \||x|^{2\alpha}D^{4}u\|_2+ C\||x|^{2\alpha-2}D^{2}u\|_2  \\
&\quad \leq \varepsilon \||x|^{2\alpha}D^{4}u\|_2 +C\left( \varepsilon \||x|^{2\alpha-1}D^{3}u\|_2+C\||x|^{2\alpha-3}Du\|_2 \right)\\
&\quad \leq \varepsilon \||x|^{2\alpha}D^{4}u\|_2 +C\left[ \varepsilon \||x|^{2\alpha-1}D^{3}u\|_2+C\left(\varepsilon \||x|^{2\alpha-1}D^{3}u\|_2+C\||x|^{2\alpha-4}u\|_2  \right) \right]
\end{align*}
therefore, one obtains \eqref{eq:2alpha-1}.

A combination of \eqref{eq:2alpha-3.1} and \eqref{eq:2alpha-1} gives \eqref{eq:2alpha-3}.

Finally, we have
\begin{align*}
 \||x|^{2\alpha-2}D^{2}u\|_2 &\leq \varepsilon \||x|^{2\alpha-1}D^{3}u\|_2+C\||x|^{2\alpha-3}Du\|_2 \\
& \leq \varepsilon  \||x|^{2\alpha}D^{4}u\|_2+ C\||x|^{2\alpha-4} u\|_2
\end{align*}
and then \eqref{eq:2alpha-2} is also proved.

Now we prove the following weighted Calder\'on-Zygmund type estimate
\begin{align}\label{eq:weightedcaldzyg}
& \||x|^{2\alpha} D^{4}u\|_2\leq C \left( \||x|^{2\alpha}\Delta^{2}u\|_{2} + \||x|^{2\alpha-4}u\|_{2} \right).
\end{align}
We have
\begin{align*}
& \||x|^{2\alpha} D^{4}u\|_{2} \leq \| D^{4}\left( |x|^{2\alpha}u \right) \|_{2}\\
&\qquad		+C\left( \| |x|^{2\alpha-1}D^{3}u\|_{2}+ \| |x|^{2\alpha-2}D^{2}u\|_{2} +\| |x|^{2\alpha-3}D  u\|_{2} +\| |x|^{2\alpha-4}u\|_{2}\right)\\
&\quad	\leq \| \Delta ^{2}\left( |x|^{2\alpha}u \right)\|_{2}\\
&\qquad		+C\left( \| |x|^{2\alpha-1}D^{3}u\|_{2}+ \| |x|^{2\alpha-2}D^{2}u\|_{2} +\| |x|^{2\alpha-3}D  u\|_{2} +\| |x|^{2\alpha-4}u\|_{2}\right)\\
&\quad	\leq \| |x|^{2\alpha}\Delta ^{2}u\|_{2}\\
&\qquad		+C\left( \| |x|^{2\alpha-1}D^{3}u\|_{2}+ \| |x|^{2\alpha-2}D^{2}u\|_{2} +\| |x|^{2\alpha-3}D u\|_{2} +\| |x|^{2\alpha-4}u\|_{2}\right)\\
&\quad	\leq \| |x|^{2\alpha}\Delta ^{2}u\|_{2} + \varepsilon \||x|^{2\alpha} D^{4}u\|+C \||x|^{2\alpha-4} u\|_2 \\
\end{align*}
and then  \eqref{eq:weightedcaldzyg} follows.

So for $h=0,1,2,3$ we have proved  that
\[
\| |x|^{2\alpha-h}D^{4-h}u\|_2\leq C\left( \| |x|^{2\alpha} \Delta ^{2} u\|_2+\||x|^{2\alpha-4 }u\|_2 \right).
\]

Now, 
from \cite[Corollary 14]{dav-hin}  for $N>8$ the higher order Rellich inequality holds
\[
\left \|\frac{u}{|x|^{4}}\right \|_{2}\leq C\|\Delta ^{2}u\|_{2}
\]
and since by the assumptions on $\alpha$ and $\beta$ one can estimate $|x|^{2\alpha-4}\leq C\left( 1+\frac{1}{|x|^{4}}+|x|^{2\beta}\right)$, it follows that  
\[
\| |x|^{2\alpha-4}u\|_{2}\leq C\left( \|\Delta^{2}u\|_{2}+ \|V^{2}u\|_{2}+\|u\|_{2}\right).
\]

Thus, for $h=0,1,2,3,4$ we have

\begin{align*}
 \| |x|^{2\alpha-h}D^{4-h}u\|_2&\leq 
	C\left( \|\Delta ^{2}u\|_2+\||x|^{2\alpha}\Delta ^{2}u\|_2+\|V^{2}u\|_2+\|u\|_{2} \right)\\
&  \leq C\left( \|Au\|_{2}+\|V^{2}u\|_{2}+\|u\|_{2} \right)
\end{align*}
and the thesis follows by Proposition \ref{pr:pot-estimate}.

\end{proof}

\begin{remark}\label{rem}
Note that for $h=0,1,2,3$ the weighted interpolation inequalities
\[
\| |x|^{2\alpha-h}D^{4-h}u\|_2\leq C\left( \| |x|^{2\alpha} \Delta ^{2} u\|_2+\||x|^{2\alpha-4 }u\|_2 \right)
\]
hold for any $u\in\ccc, N\geq1$, and any real $\alpha$. In particular, setting $\alpha=0$ the inequalities read as
\begin{equation*}
\| |x|^{-h}D^{4-h}u\|_2\leq C\left( \|   \Delta ^{2} u\|_2+\||x|^{ -4 }u\|_2 \right)
\end{equation*}
then, if $N>8$, by the higher order Rellich's inequality one obtains
\begin{equation*}\label{inter}
 \| |x|^{-h}D^{4-h}u\|_2\leq C  \|   \Delta ^{2} u\|_2 
\end{equation*}
for $h=0,1,2,3,4$.

Moreover, for $u\in H^4(\z)$,  reasoning as above with the standard interpolation inequality $||D u||_2\leq C\left(|| D^2u||_2+ ||u||_2\right)$ one obtains that $||u||_{H^4(\z)}\leq C\left(||D^4 u||_2+ ||u||_2\right)$ and then for $u\in D(A)$
\begin{equation}\label{normah4}
 ||u||_{H^4(\z)}\leq C||u||_A.
\end{equation}
\end{remark}

Now arguing as in Proposition \ref{pr:core-for-a} we prove that $C_{c}^{\infty}(\R^{N}\setminus \{0\})$ is a core for $A$ in
\[
D_{2}(A)=\{u\in H^{4}(\R^{N})\,:\,V^{2}u\in L^{2}(\R^{N}), |x|^{2\alpha-h}D^{4-h}u\in L^{2}(\R^{N}) \text{ for } h=0,1,2,3,4\}.
\]

\begin{proposition}\label{pr:core-for-A2} Let $N>8,\alpha>0,\beta>(\alpha-2)^+$. Then
$C_{c}^{\infty}(\R^{N})\setminus \{0\}$ is dense in $D_{2}(A)$ with respect to the operator norm.
There exists $C\geq 0$ such that for every $u\in D_{2}(A)$  
\begin{align*}
& \||x|^{2\alpha-h}D^{4-h}u\|_{2}\leq C\left(\|Au\|_{2}+\|u\|_{2} \right)\text{ for }h=0,1,2,3,4,\\
& \|V^{2}u\|_{2}\leq C\left( \|Au\|_{2}+\|u\|_{2}\right).
\end{align*}
\end{proposition}
\begin{proof}
It is enough to prove that the set of functions in $H^4(\R^N)$ with compact support contained in 
$\rn\setminus\{0\}$ is dense in $D_{2}(A)$ with respect to the operator norm.
Take $u\in D_{2}(A)$ and consider $u_n=u\varphi _n$,
where $\varphi_n\in 
\ccc$ is such that 

\begin{equation*}  
\left\{
\begin{array}{ll}
\varphi_n=0  \  \text{in} \   B({\frac1n})\cup B^c({2n}), \\
\varphi_n=1 \  \text{in} \  B(n)\setminus B({\frac2n)}, \\
0\leq \varphi_n\leq 1,\\
|\nabla \varphi_n(x)|\leq C\frac{1}{|x|},\\
|D^{2} \varphi_n(x)|\leq C\frac{1}{|x|^{2}},\\
|D^{3} \varphi_n(x)|\leq C\frac{1}{|x|^{3}},\\
|D^{4} \varphi_n(x)|\leq C\frac{1}{|x|^{4}}.\\
\end{array}\right.
\end{equation*}  
 Observe that here we can consider the same sequence as in Proposition \ref{pr:core-for-a}.

We have
\begin{align*}
\Delta^{2}(u\varphi_{n})&=\Delta \left( \Delta u \varphi_{n}+2\nabla u\cdot \nabla \varphi_{n}+u\Delta \varphi_{n} \right) \\
&  =\varphi_{n}\Delta^{2}u +u\Delta ^{2}\varphi_{n}+4\nabla\Delta u\cdot \nabla \varphi_{n}+4\nabla u\cdot \nabla \Delta\varphi_{n}\\
&\quad +4\sum_{i=1}^{N}\nabla D_{i}u\cdot \nabla D_{i}\varphi_{n}+2\Delta  u\Delta \varphi_{n}.
\end{align*} 
Then
\begin{align*}
 |Au_{n}(x)-Au(x)|&\leq  |1-\varphi_{n}(x)||Au(x) |\\
&\quad +C\sum_{i,j,k=1}^N \left(|x|^{-1}+|x|^{\alpha-1}+|x|^{2\alpha-1}\right)|  D_{ijk}u(x)|\chi_{K_{n}}
	\\&\quad+C\sum_{i,j=1}^N \left(|x|^{-2}+|x|^{\alpha-2}+|x|^{2\alpha-2}\right)|D_{ij}u(x)|\chi_{K_{n}}\\
&\quad +C\sum_{i=1}^N \left(|x|^{-3}+|x|^{\alpha-3}+|x|^{2\alpha-3}\right)|D_{i}u(x)|\chi_{K_{n}}
	\\&\quad+C \left(|x|^{-4}+|x|^{\alpha-4}+|x|^{2\alpha-4}\right)|u(x)|\chi_{K_{n}}\\	
&   \leq  | 1-\varphi_{n}(x)|| Au(x) |\\
&\quad +C\sum_{i,j,k=1}^N \left(|x|^{-1}+|x|^{2\alpha-1}  \right)|D_{ijk}u(x)|\chi_{K_{n}}
\\&\quad	+C\sum_{i,j=1}^N \left(|x|^{-2}+|x|^{2\alpha-2}  \right)|D_{ij}u(x)|\chi_{K_{n}}\\
&\quad +C\sum_{i=1}^N\left(  |x|^{-3}+ |x|^{2\alpha-3} \right)|D_{i}u(x)|\chi_{K_{n}}
	\\&\quad+C \left(|x|^{-4}+ |x|^{2\alpha-4} \right)|u(x)|\chi_{K_{n}}
\end{align*}
where $K_n=B({\frac2n)}\setminus B({\frac1n)}\cup B({2n})\setminus B(n)$. All the terms in the right hand side converge  to $0$
pointwisely and hence in $L^{2}(\R^{N})$ since $|x|^{2\alpha -h}D^{4-h}u(x)\chi_{K_{n}}\leq |x|^{2\alpha -h}D^{4-h}u(x)\in L^{2}\left( \R^{N} \right)$
for $h=0,1,2,3,4$ and by Remark \ref{rem} also for $\alpha=0$.
\end{proof}

Finally we prove that $D(A)$ coincides with $D_{2}(A)$.

\begin{theorem}\label{th:domain}
Assume that $N>8,\alpha>0,\beta>(\alpha-2)^+$. Then maximal domain $D(A)$  coincides with $D_2(A)$.
\end{theorem}
\begin{proof}
We have to prove only the inclusion $D(A)\subset D_{2}(A)$.

Let $\tilde u\in D(A)$ and set $f=A\tilde u  +\lambda\tilde u $.
The operator $A$ in 
$C\left(\frac{1}{\rho},\rho\right):=B(\rho)\setminus B\left(\frac{1}{\rho}\right)$, $\rho>0$,  is an elliptic operator with bounded coefficients, then for a suitable $\lambda$ the problem
\begin{equation*}\label{palla0}
\left\{
\begin{array}{ll}
Au+\lambda u=f&\text{ in }C\left(\frac{1}{\rho},\rho\right),\\
u=0&\text{ on }\partial C\left(\frac{1}{\rho},\rho\right),\\
D u=0&\text{ on }\partial C\left(\frac{1}{\rho},\rho\right),
\end{array}
\right.
\end{equation*}
admits a unique solution $u_\rho$ in $W^{4,2}\left(C\left(\frac{1}{\rho},\rho\right)\right)\cap W_0^{1,2}\left(C\left(\frac{1}{\rho},\rho\right)\right)$ 
 (cf. \cite[Section 3.2]{luna}). 
Now $u_\rho\in D_2(A)$ and by Proposition \ref{pr:core-for-A2} and \eqref{normah4}
\begin{align*}
&\||x|^{2\alpha}D^4 u_\rho\|_{L^2(B(\rho))}+\| |x|^{2\alpha-1}D^3 u_\rho\|_{L^2(B(\rho))}+\| |x|^{2\alpha-2} D^2 u_\rho\|_{L^2(B(\rho))}+\||x|^{2\alpha-3}D u_\rho\|_{L^2(B(\rho))}\\
&\quad\quad +\||x|^{2\alpha-4}u_\rho\|_{L^2(B(\rho))}+\|V^2u_\rho\|_{L^2(B(\rho))}+||u||_{H^4(B(\rho))}  
 \le C\left(\|Au_\rho\|_{2}+\| u_\rho\|_{2}\right).
\end{align*}
By a standard weak compactness argument it is possible to construct a sequence
$(u_{\rho_n})$ which converges to a function $u$ in $W^{4,2}_{	\rm loc}(\z)$ such that $Au+\lambda u=f$.
Since the estimates above are independent of $\rho$, also $u\in D_2(A)$.
Then we have $A\tilde u+\lambda \tilde u=Au+\lambda u$ and since $D_2(A)\subset D(A)$ and $A+\lambda$ is invertible on $D(A)$  being the generator of a  $C_0$-semigroup, it is possible to conclude that $\tilde u=u$.

\end{proof}

In the last part of this section we investigate on the spectrum of the operator $A$. 
\begin{proposition}\label{spectrum} For any $\alpha>0$, $\beta>(\alpha-2)^+$, and $N>8$ the spectrum of $A$ consists of a sequence of negative real eigenvalues which accumulates at $-\infty$.
\end{proposition}
\begin{proof} We prove that $D(A)$ is compactly embedded into $L^{2}\left( \R^{N} \right)$, this will prove that the spectrum of $A$ consists of eigenvalues.
By  Proposition \ref{pr:core-for-A2} and Theorem \ref{th:domain} we have
\[
\int_{\R^{N}} |x|^{4\beta}|u^{2}|dx \leq C\left( \|Au\|_{2}+\|u\|_{2}\right).
\]
Then we consider the unitary ball of $D(A)$ that is ${\mathcal B}=\{u\in D(A)\,|\, \|Au\|_{2}+\|u\|_{2}\leq 1\}$ and for any $u\in \mathcal B$ we have
\[
\int_{\R^{N}} |x|^{4\beta}|u^{2}|dx \leq C.
\]
Now fix $\varepsilon>0$, there exists $M>0$ such that $\frac{C}{M^{4\beta}}<\varepsilon^{2} $, so that
\begin{align*}
& \int_{|x|>M} |u^{2}|dx\leq \frac{1}{M^{4\beta}} \int_{|x|>M} |x|^{4\beta}|u^{2}|dx\leq \frac{C}{M^{4\beta}} < \varepsilon^{2}. \\
\end{align*}
Let us now consider the set of the restriction to $B(M)$, the ball of center 0 and radius $M$, of the functions in $\mathcal B$ that we denote by ${\mathcal B}_M$.
We observe that for every $u\in D(A)$,   we have $\|u\|_{H^{4}\left( \R^{N} \right)}\leq C\left( \|Au\|_{2}+\|u\|_{2} \right)$ 
then ${\mathcal B}_M$ is bounded in $H^{4}\left( B(M) \right)$.
Since  $H^{4}(B(M))$ is compactly embedded into $L^{2}(B(M))$  we have that ${\mathcal B}_M$ is totally bounded in 
$L^{2}\left( B(M) \right)$. Then there exists $u_{1},\dots ,u_{n}\in L^{2}\left( B(M) \right) $ such that
\[
{\mathcal B}_M\subset \cup_{i=1}^{n}\{u\in L^{2}(B(M))\,\\, \|u-u_{i}\|_{L^{2}(B(M))}<\varepsilon \}.
\]
Now, for $i=1,\dots n$, we set 
$\tilde u_{i}(x)=u_{i}(x)$ if $|x|\leq M$ and $\tilde u_{i}(x)=0$  otherwise. For every $u\in {\mathcal B}$    
we have $\|u-\tilde u_{i}\|_{2}\leq 2\varepsilon$ and
\[
{\mathcal B} \subset \cup_{i=1}^{n}\{u\in L^{2}(\R^{N})\,\\, \|u-\tilde u_{i}\|_{2}<2\varepsilon \}.
\]
So $\mathcal B$ is totally bounded in $L^{2}(\R^{N})$.

Now let $\lambda \in \C $ be an eigenvaule of $A$, there exists $u\in D(A)$ such that $Au+\lambda u=0$ and then 
\[
\Delta^{2}u+\frac{V^{2}}{a^{2} }u+\lambda \frac{u}{a^{2}}=0.
\]
Multiplying by $u$ and integrating by parts we obtain, recalling that $u\in H^4(\z),$
\[
\int_{\R^{N}}\left ((\Delta u)^{2}+\frac{V^{2}}{a^{2} }u^{2}\right)dx+\lambda \int_{\R^{N}} \frac{u^{2}}{a^{2}}dx=0
\]
from which follows that $\lambda$ is real and negative.
\end{proof}

\bibliographystyle{amsplain}

\bibliography{bibfile}

\providecommand{\bysame}{\leavevmode\hbox to3em{\hrulefill}\thinspace}
\providecommand{\MR}{\relax\ifhmode\unskip\space\fi MR }
\providecommand{\MRhref}[2]{%
  \href{http://www.ams.org/mathscinet-getitem?mr=#1}{#2}
}
\providecommand{\href}[2]{#2}
\begin{thebibliography}{10}

\bibitem{adams}
D.~Adams, \emph{${L}^p$ potential theory techniques and nonlinear {PDE}},
  Potential theory (Nagoya, 1990), 1–15, de Gruyter, Berlin, 1992.

\bibitem{agrt}
D.~Addona, F.~Gregorio, A.~Rhandi, and C.~Tacelli, \emph{Bi-{K}olmogorov type
  operators and weighted {R}ellich's inequalities}, Nonlinear Differ. Equ.
  Appl. \textbf{29} (2022), no.~2, Paper No. 13, 37 pp.

\bibitem{agm}
S.~Agmon, \emph{The ${L}^p$ approach to the {D}irichlet problem. {I}.
  {R}egularity theorems.}, Ann. Scuola Norm. Sup. Pisa Cl. Sci. \textbf{13}
  (1959), no.~3, 405–448.

\bibitem{BelhajAli}
S.~Belhaj Ali, \emph{Elliptic operators with unbounded diffusion and singular
  lower order terms in ${L}^p$-spaces}, Semigroup Forum \textbf{98} (2019),
  499--520.

\bibitem{antman}
S.~S. Antman, \emph{Nonlinear problems of elasticity, 2nd edn}, Applied
  Mathematical Sciences, p. 107. Springer, New York, 2005.

\bibitem{boutiah-et-al1}
S.~E. Boutiah, F.~Gregorio, A.~Rhandi, and C.~Tacelli, \emph{Elliptic operators
  with unbounded diffusion, drift and potential terms}, Journal of Differential
  Equations. \textbf{264} (2018), no.~3, 2184--2204.

\bibitem{boutiah-et-al2}
S.~E. Boutiah, A.~Rhandi, and C.~Tacelli, \emph{Kernel estimates for elliptic
  operators with unbounded diffusion, drift and potential terms}, Discrete
  Contin. Dyn. Syst. Ser. A. \textbf{39} (2019), no.~2, 803--817.

\bibitem{bcgt}
S.E. Boutiah, L.~Caso, F.~Gregorio, and C.~Tacelli, \emph{Some results on
  second-order elliptic operators with polinomially growing coefficients in
  ${L}^p$-spaces}, J. Math. Anal. Appl. \textbf{501} (2021), 21 pp.

\bibitem{can-gre-rhan-tac}
A.~Canale, F.~Gregorio, A.~Rhandi, and C.~Tacelli, \emph{Weighted {H}ardy’s
  inequalities and {K}olmogorov type operators}, Applicable Analysis
  \textbf{98} (2019), no.~7, 1236--1254.

\bibitem{can-rhan-tac1}
A.~Canale, A.~Rhandi, and C.~Tacelli, \emph{Schr{\"o}dinger type operators with
  unbounded diffusion and potential terms}, Ann. Sc. Norm. Super. Pisa Cl. Sci.
  \textbf{XVI} (2016), no.~2, 581--601.

\bibitem{can-rhan-tac2}
\bysame, \emph{Kernel estimates for {S}chr{\"o}dinger type operators with
  unbounded diffusion and potential terms}, Journal of Analysis and its
  Applications \textbf{36} (2017), no.~4, 377--392.

\bibitem{can-tac1}
A.~Canale and C.~Tacelli, \emph{Kernel estimates for a {S}chr{\"o}dinger type
  operator}, Riv. Mat. Univ. Parma \textbf{7} (2016), 341--350.

\bibitem{davies95a}
E.~B. Davies, \emph{Uniformly elliptic operators with measurable coefficients},
  J. Funct. Anal. \textbf{132} (1995), 141–169.

\bibitem{dav-hin}
E.~B. Davies and A.~M. Hinz, \emph{Explicit constants for {R}ellich
  inequalities in ${L}^p({\Omega})$}, Math. Z. \textbf{227} (1998), no.~3,
  511--523.

\bibitem{dav-hinK}
\bysame, \emph{Kato class potentials for higher order elliptic operators}, J.
  London Math. Soc. (2) \textbf{58} (1998), no.~3, 669–678.

\bibitem{dur-man-tac1}
T.~Durante, R.~Manzo, and C.~Tacelli, \emph{Kernel estimates for
  {S}chr{\"o}dinger type operators with unbounded coefficients and singular
  potential terms}, Ricerche di Matematica \textbf{65} (2016), no.~1, 289--305.

\bibitem{for-gre-rha}
S.~Fornaro, F.~Gregorio, and A.~Rhandi, \emph{Elliptic operators with unbounded
  diffusion coefficients perturbed by inverse square potentials in
  ${L}^p$-spaces}, Communications on Pure \& Applied Analysis \textbf{15}
  (2016), 2357--2372.

\bibitem{for-lor}
S.~Fornaro and L.~Lorenzi, \emph{Generation results for elliptic operators with
  unbounded diffusion coefficients in ${L}^p$ and ${C}_b$-spaces}, Discr. Cont.
  Dyn. Syst. A \textbf{18} (2007), 747--772.

\bibitem{gal-kam}
V.~A. Galaktionov and I.~V. Kamotski, \emph{On nonexistence of
  {B}aras-{G}oldstein type for higher-order parabolic equations with singular
  potentials}, Trans. Amer. Math. Soc. \textbf{362} (2010), no.~8, 4117–4136.

\bibitem{gre-ker}
F.~Gregorio and J.~Kerner, \emph{On {L}ennard-{J}ones-type potentials on the
  half-line}, Arch. Math. \textbf{112} (2019), no.~1, 101--111.

\bibitem{gre-mil}
F.~Gregorio and S.~Mildner, \emph{Fourth-order {S}chr{\"o}dinger type operator
  with singular potentials}, Arch. Math. \textbf{107} (2016), 285--294.

\bibitem{liu-don}
Y.~Liu and J.~Dong, \emph{Some estimates of higher order {R}iesz transform
  related to {S}chr{\"o}dinger type operators}, J. Potential Anal. \textbf{32}
  (2010), 32--41.

\bibitem{lor-rhan}
L.~Lorenzi and A.~Rhandi, \emph{On {S}chr{\"o}dinger type operators with
  unbounded coefficients: generation and heat kernel estimates}, J. Evol. Equ.
  \textbf{15} (2015), 53--88.

\bibitem{luna}
A.~Lunardi, \emph{Analytic semigroups and optimal regularity in parabolic
  problems}, Birkhauser, 1995.

\bibitem{mele}
V.V. Meleshko, \emph{Selected topics in the history of the two-dimensional
  biharmonic problem}, Appl. Mech. Rev. \textbf{56} (2003), 33--85.

\bibitem{met-oka-sob-spi}
G.~Metafune, N.~Okazawa, M.~Sobajima, and C.~Spina, \emph{Scale invariant
  elliptic operators with singular coefficients}, J. Evol. Equ. \textbf{16}
  (2016), 391--439.

\bibitem{Met-Sob-Spi2}
G.~Metafune, M.~Sobajima, and C.~Spina, \emph{Weighted {C}alder\'on–{Z}ygmund
  and {R}ellich inequalities in ${L}^p$}, Math. Ann. \textbf{361} (2015),
  313–366.

\bibitem{met-spi2}
G.~Metafune and C.~Spina, \emph{Elliptic operators with unbounded diffusion
  coefficients in ${L}^p$ spaces}, Ann. Sc. Norm. Super. Pisa Cl. Sci. (5)
  \textbf{11} (2012), no.~2, 303--340.

\bibitem{met-spi-tac}
G.~Metafune, C.~Spina, and C.~Tacelli, \emph{Elliptic operators with unbounded
  diffusion and drift coefficients in ${L}^p$ spaces}, Adv. Diff. Equat
  \textbf{19} (2014), no.~5-6, 473--526.

\bibitem{met-spi-tac2}
\bysame, \emph{On a class of elliptic operators with unbounded diffusion
  coefficients}, Evol. Equ. Control Theory \textbf{3} (2014), no.~4, 671--680.

\bibitem{ouhabaz}
E.M. Ouhabaz, \emph{Analysis of heat equations on domains}, London Math. Soc.
  Monogr. Ser., {\bf 31}, Princeton Univ. Press, 2004.

\bibitem{sgk}
T.~Sedrakyan, L.~Glazman, and A.~Kamenov, \emph{Absence of bose condensation on
  lattices with moat bands}, Phys. Rev. B \textbf{89} (2014), 201112.

\bibitem{sugano}
S.~Sugano, \emph{${L}^p$ {E}stimates for {S}ome {S}chr{\"o}dinger {T}ype
  {O}perators and a {C}alderón-{Z}ygmund {O}perator of {S}chr{\"o}dinger
  {T}ype}, Tokyo Journal of Mathematics \textbf{30} (2007), no.~1, 179--197.

\end{thebibliography}

\end{document}